\documentclass[11pt,a4paper]{article}

\usepackage[english]{babel}
\usepackage{amssymb}
\usepackage{amsmath}

\usepackage[margin=1.6in]{geometry}

\mathcode`l="8000
\begingroup
\makeatletter
\lccode`\~=`\l
\DeclareMathSymbol{\lsb@l}{\mathalpha}{letters}{`l}
\lowercase{\gdef~{\ifnum\the\mathgroup=\m@ne \ell \else \lsb@l \fi}}%
\endgroup

\renewcommand{\mid}{:}

\makeatletter
\newcommand\need[1]{\par \penalty-100 \begingroup %
   \dimen@\pagegoal \advance\dimen@-\pagetotal %
   \ifdim #1>\dimen@ %
      \ifdim\dimen@>\z@ \vskip -\pagedepth plus 1fil \fi
      \break
   \fi \endgroup}
\makeatother

\usepackage[style=math, giveninits=true,dashed=false]{biblatex}
\usepackage[thmmarks,hyperref]{ntheorem}
\theoremstyle{plain}
\theoremheaderfont{\normalfont\bfseries}
\theorembodyfont{\slshape}
\theoremseparator{.}
\theorempreskip{2ex}
\theorempostskip{\topsep}
\theoremnumbering{arabic}
\theoremsymbol{}
\newtheorem{theorem}{Theorem}

\newtheorem{corollary}{Corollary}
\newtheorem{remark}{Remark}
\newtheorem{lemma}{Lemma}
\newtheorem{cl}{Claim}
\theorembodyfont{\normalfont}
\newtheorem{definition}{Definition}

\makeatletter
\newtheoremstyle{proof}%
{\item[\hskip\labelsep \theorem@headerfont ##1\theorem@separator]}%
{\item[\hskip\labelsep \theorem@headerfont ##1\ ##3\theorem@separator]}
\makeatother

\theoremheaderfont{\normalfont\bfseries}
\theorembodyfont{\upshape}

\theoremstyle{proof}

\theorempreskip{3.5ex}
\theorempostskip{\topsep}
\theoremseparator{.}

\theoremsymbol{\ensuremath{_\Box}}

\newtheorem{proof}{Proof}

\theorempreskip{0ex}
\theoremheaderfont{\itshape}
\theoremsymbol{\ensuremath{_\Diamond}}

\newtheorem{smallproof}{Proof}

\usepackage{hyperref}
\hypersetup{
    colorlinks=true,
    linkcolor=black,
    citecolor=black,
    pdftitle={Kempe Chains and Rooted Minors},
    pdfauthor={Kriesell, Mohr},
    bookmarks=true,
    pdfpagemode=UseOutlines,
}

\usepackage{scalerel}

\makeatletter
\newcommand{\hourclass}{\mathbin{\scalerel*{\@hgpic}{\ensuremath{\Sigma}}}}%
\newcommand{\@hgpic}{%
    \setlength{\unitlength}{0.34cm}%
    \begin{picture}(1,1.5)%
    \thicklines%
    \put(0,0){\line(2,3){1}}%
    \put(1,1.5){\line(-1,0){1}}%
    \put(0,1.5){\line(2,-3){1}}%
    \put(1,0){\line(-1,0){1}}%
    \end{picture}%
}
\makeatother

\begin{document}

\title{\textbf{Kempe Chains and Rooted Minors}}
\author{Matthias Kriesell \and Samuel Mohr\thanks{Gefördert durch die Deutsche Forschungsgemeinschaft (DFG) -- 327533333.}}

\maketitle

\begin{abstract}
  \setlength{\parindent}{0em}
  \setlength{\parskip}{1.5ex}
\noindent
  A (minimal) {\em transversal} of a partition is a set which contains exactly one element from each member of the partition and nothing else. A \emph{coloring} of a graph is a partition of its vertex set into anticliques, that is, sets of pairwise nonadjacent vertices. 
  We study the following problem: Given a transversal $T$ of a proper coloring $\mathfrak{C}$ of some graph $G$,
  is there a partition $\mathfrak{H}$ of a subset of $V(G)$ into connected sets such that $T$ is a transversal of $\mathfrak{H}$
and such that two sets of $\mathfrak{H}$ are adjacent  if their corresponding vertices from $T$ are connected by a path in $G$ using only two colors?

  It has been suggested by the first author to study the following question: for any transversal $T$ of a coloring $\mathfrak{C}$ of order $k$ of some graph $G$ such that any
  pair of color classes induces a connected graph, does there exist such a partition $\mathfrak{H}$ with pairwise adjacent sets (which would prove Hadwiger's Conjecture for the class of uniquely optimally colorable graphs)?
  This is open for small $k \geq 5$, here we give a proof for the case that $k=5$ and the subgraph induced by $T$ is connected. 
Moreover, we show that for $k\geq  7$, it is not sufficient for the existence of  $\mathfrak{H}$ as above just to force any two transversal vertices to be connected by a 2-colored path. 

\medskip


\end{abstract}

\section{Introduction}

Hadwiger's Conjecture states that the order of a largest clique minor of every graph $G$ is at least the minimum number of colors needed to color $G$ properly~\cite{Hadwiger1943}.
Introduced as a possible approach to solve the Four-Color-Theorem in 1943, it turned out to be one of the most challenging conjectures in structural graph theory. 
A breakthrough result on the general conjecture is due to Robertson, Seymour, and Thomas solving another step Hadwiger's Conjecture~\cite{robertson1993hadwiger}: a graph with chromatic number $k\leq 6$ contains a clique minor of order $k$. 
In the 1980s, Kostochka~\cite{kostochka1984lower} and Thomason~\cite{thomason1984extremal} independently proved the following: every $K_t$-minor-free graph  has average degree $O(t \sqrt{\log t})$ and hence is $O(t \sqrt{\log t})$-colorable.
Very recently, Delcourt and Postle improved the order of magnitude and showed that a $K_t$-minor-free graph is even $O(t\log\log t)$-colorable~\cite{delcourt2021reducing}.

Another natural approach to tackle Hadwiger's Conjecture is to bound the  size of color classes, which however has also not been crowned with large success (cf.~\cite{seymour2016hadwiger}). 
The first non-trivial step is to forbid antitriangles, which bounds the sizes of color classes by $2$. 
Seymour conjectured that Hadwiger's Conjecture holds in this particular case (cf.~\cite{blasiak2007special}) but likewise Hadwiger's Conjecture, it is also still unknown. 
Moreover, Seymour suggested to look for counterexamples to Hadwiger's Conjecture within this class of graphs~\cite{seymour2016hadwiger}. 
The fact that the longstanding conjecture of Hadwiger has barely generated results is an indicator that it is tenacious.

In this paper, we study Hadwiger's Conjecture under the assumption that more details about the coloring/colorings are known. Intuitively, the hope would be that it leads to a more rigid structure, which will turn out to be benefitial for constructing minors. 
Before we motivate this approach further and provide details, it is necessary to clarify the notation used throughout this paper. 

\subsection*{Notation}

All graphs in the paper are finite, undirected, and simple. 
For terminology not defined here we refer to contemporary text books such as~\cite{BondyMurty2008} or~\cite{Diestel2017}.
By $K_S$ the complete graph on a finite set $S$ is denoted. Note that we emphasize the vertex set $S$ here and don't speak about $K_{|S|}$. 
A (minimal) {\em transversal} of a set $\mathfrak{C}$ of disjoint sets is a set $T$ containing exactly one member of every $A \in \mathfrak{C}$
and nothing else; we also say that $\mathfrak{C}$ is {\em traversed} by $T$.
A {\em coloring} of a graph $G$ is a partition $\mathfrak{C}$ of its vertex set $V(G)$ into {\em anticliques},
that is, sets of pairwise nonadjacent vertices. 
The {\em chromatic number} $\chi(G)$ is the smallest order of a coloring of $G$.
A {\em Kempe chain} is a component of $G[A \cup B]$
for some $A \not= B$ from a coloring $\mathfrak{C}$. 
A graph $H$ is a {\em minor} of a graph $G$ if there exists a family $c=(V_t)_{t \in V(H)}$ of pairwise disjoint subsets of $V(G)$,
called {\em bags}, such that $V_t$ is nonempty and $G[V_t]$ is connected for all $t \in V(H)$ and there is an edge connecting $V_s$ and $V_t$ for all $st \in E(H)$.
Any such $c$ is called an {\em $H$\!-certificate} in $G$,
and a {\em rooted $H$\!-certificate} if, moreover, $V(H) \subseteq V(G)$ and $t \in V_t$ for all $t \in V(H)$.
If there exists a rooted $H$\!-certificate, then $H$ is a {\em rooted minor} of $G$.

\subsection*{Uniquely optimally colorable graphs}

As a new approach, the first author suggested to study graphs with a bounded number of optimal colorings and, in particular, investigate uniquely optimally colorable graphs~\cite{Kriesell2017}. 
Let $x_1,\dots, x_k$ be differently colored vertices of a uniquely $k$-colorable graph, then it is straightforward to see that there exists a system of edge disjoint $x_i$,$x_j$-paths ($i\neq j$ from $\{1,\dots, k\}$) by taking only edges between the corresponding two color classes. This is a so-called clique immersion of order $k$ at $x_1,\dots, x_k$ and the question suggests itself whether there exists a clique minor of order $k$ rooted at $x_1,\dots, x_k$. 
This question has been answered affirmatively if $k\leq 10$~\cite{kriesell2021note}, if the graph is antitriangle-free~\cite{Kriesell2017},  and the question is moreover true for line graphs~\cite{kriesell2019rooted}.

The idea of taking a vertex set of a graph containing vertices of all colors is not new. 
Holroyd stated the following conjecture~\cite{holroyd1997strengthening}: if a graph $G$ is optimally $k$-colorable and a set $S \subseteq V(G)$ takes all $k$ colors in every $k$-coloring of $G$, then there exists a $K_k$-certificate such that every bag intersects $S$. Note that $|S|$ can be larger than $k$ in this conjecture. 
Holroyd called it the Strong Hadwiger's Conjecture; taking $S=V(G)$ implies Hadwiger's Conjecture. 
In the setting of uniquely optimally colorable graphs the Strong Hadwiger's Conjecture is equivalent to the above mentioned question.

\subsection*{Relaxation to two-colored path systems}

In many cases, the object under consideration are clique minors, i.\,e., a $K_k$-certificate for some $k$ in the graph. 
In the present paper, we study the following question, which does not necessarily ask for a clique minor:
let $G$ be a graph and assume a coloring of $G$ with $k$ colors ($k\in\mathbb{N}$) is given. Moreover, let $x_1,\dots, x_k$ be vertices of $G$ all colored differently. 
Then there is a  system of edge disjoint $x_i$,$x_j$-paths for some but not necessarily all pairs $(x_i,x_j)$ by only taking edges between the two corresponding color classes. 
We may then ask whether $G$ contains a graph $H$ on vertices $x_1,\dots, x_k$ as a rooted minor, where the edge set of $H$ is defined in the obvious way (the edge $x_ix_j$ is an edge in $H$ if an $x_i$,$x_j$-path in the path system in $G$ exists).

To be more precice, we define the graph that we would like to obtain as a minor. This graph should capture the property of the presence of two-colored paths between the transversal vertices of a coloring. 
For a transversal $T$ of a coloring $\mathfrak{C}$ of a graph $G$ we define the {\em routing graph $H(G,\mathfrak{C},T)$}
to be the graph on the vertex set $T$ where any two $s \not= t$ from $T$ are adjacent if and only if they belong to the same Kempe chain in $G$.

\subsection*{Property (*)}

We study the following problem: what are the graphs $H$ such that if $H$ is a routing graph of some graph $G$ with coloring $\mathfrak{C}$ and a transversal $T$, then $G$ has a rooted $H$\!-certificate. 
Those graphs are said to have a special property, which we call property (*).

\begin{definition}
Let us say that a graph $K$ {\em has property (*)} if for every transversal $T$ of every coloring $\mathfrak{C}$ with $|\mathfrak{C}|=|V(K)|$ of every graph $G$ such that $K$ is
isomorphic to a spanning subgraph $H$ of $H(G,\mathfrak{C},T)$,  there exists a rooted $H$\!-certificate in $G$ (rooted in $T$). 
\end{definition}

It is obvious that property (*) holds for $K_1$ and transfers to isomorphic copies of $K$. We will show later, that 
property (*) inherits to subgraphs of $K$ and that $K$ has property (*) if and only if every component of $K$ has.

\subsection*{Kempe colorings}

We start by investigating the setting in the case that the routing graph is a complete graph. 
Uniquely colorable graphs imply immediately the presence of all two-colored paths; however, there is a less restrictive class of graphs capturing this property. These graphs are those having a Kempe coloring. 

A coloring $\mathfrak{C}$ is a {\em Kempe coloring} if any two vertices from distinct color classes belong to the same Kempe chain or,
in other words, the union of any two color classes is connected.
The first author asked in~\cite{Kriesell2017} whether given
  a Kempe coloring $\mathfrak{C}$ of some graph $G$ and a transversal $T$ of $\mathfrak{C}$,
  there exists a set of connected, pairwise disjoint, pairwise adjacent subsets of $V(G)$ traversed by $T$.

This would prove Hadwiger's Conjecture --- that every graph with chromatic number $k$ has a clique minor of order $k$~\cite{Hadwiger1943} ---
for graphs with a Kempe coloring, in particular for uniquely $k$-colorable graphs.
In the terminology defined above, the conjecture reads as follows:
If $\mathfrak{C}$ is a Kempe coloring of $G$ and $T$ is a transversal of $\mathfrak{C}$, then $H:=H(G,\mathfrak{C},T)$ is the complete graph on $T$, i.\,e.\ $K_T=H$,
and there exists a rooted $H$\!-certificate in $G$. This would follow if every complete graph --- and hence {\em every} graph --- had property (*).

However, property (*) is known to be too restrictive to be true:
B{\"o}hme, Kostochka, and Thomason constructed graphs on $2k$ vertices with a $k$-Kempe coloring such that the largest clique minor is of order $\frac23k+o(k)$ \cite{bohme2010hadwiger} (see also~\cite{bonamy2021recolouring}).
This immediately implies that a Kempe coloring (which is not an optimal coloring) is not enough to guarantee a large clique minor; and hence $K_k$ must fail to have property (*) for large $k$. 
For small $k$ and for completeness, 
we will see that already $K_7$ does not have property (*).

\subsection*{Graphs with property (*)}

We emphasized that it makes a difference whether the subgraph induced by two color classes is connected or just both transversal vertices belong to the same Kempe chain. 
Fabila-Monroy and Wood studied graphs containing a rooted $K_4$-minor~\cite{FabilaMonroyWood2013}. They characterized those graphs by the presence of a system of edge disjoint paths instead of considering color classes and Kempe chains. 
It is possible to derive from their result that graphs with at most four vertices do have property (*) (cf.~Theorem~\ref{T3}).
Therefore, the question for the largest
$b$ such that all graphs of order at most $b$ have property (*) suggests itself (it must be one of $4,5,6$ by Theorem~\ref{T1}).

Another natural question to ask is whether the number of graphs having property (*) is bounded. 
This is not the case,
as graphs with at most one cycle have property (*) by Theorem~\ref{T4}. 
As a consequence, for example, we get
that if $x_0,\dots,x_{\ell-1}$ belong to different color classes of some coloring $\mathfrak{C}$ of
a graph $G$ and $x_i,x_{i+1}$ belong to the same Kempe chain for $i \in \{0,\dots,\ell-1\}$ (indices modulo $\ell$), then
there exists a cycle $C$ in $G$ and disjoint $x_i$,$y_i$-paths $P_i$ with $V(C) \cap V(P_i)=\{y_i\}$ and $y_0,\dots,y_{\ell-1}$ occur
in this order on $C$.

Apart from this, we determine a number of further $5$-vertex graphs having property (*) and infer that
if $T$ is a {\em connected} transversal of a Kempe coloring of order $5$ of some graph $G$, then
there exists a rooted $H(G,\mathfrak{C},T)$-certificate (where $H(G,\mathfrak{C},T)$ is isomorphic to $K_5$).

\section{Subgraphs of graphs with property (*)}

In this section, we show that property (*) inherits to subgraphs and present the tool of extending certificates from graphs obtained by some contractions in large details. This tool will be seen a couple of times throughout the following proofs in this paper. 

\begin{theorem}\label{Tsubgraphs}
Property (*) inherits to subgraphs of $K$; furthermore, $K$ has property (*) if and only if every component of $K$ has.
\end{theorem}

\begin{proof}
First, assume $K'$ has property (*) and $K$ is a spanning subgraph of $K'$. Take an arbitrary graph $G$ with a coloring $\mathfrak{C}$ with $|\mathfrak{C}|=|V(K)|$ and  a transversal $T$ of $\mathfrak{C}$ such that $K$ is
isomorphic to a spanning subgraph $H$ of $H(G,\mathfrak{C},T)$. For $e\in E(K')\setminus E(K)$ add a suitable edge between two transversal vertices to $G$ (if not already present) to obtain a graph $G'$. Then $K'$ is isomorphic to a spanning subgraph $H'$ of $H(G',\mathfrak{C},T)$.
Since $K'$ has property (*), there is a rooted $H'$\!-certificate $c$ in $G'$ and $c$ is also a rooted $H$\!-certificate in $G$. 

Next, assume that $K'$ has property (*) and let $K\neq K'$ be a component of $K'$. Take an arbitrary graph $G$ with a coloring $\mathfrak{C}$ with $|\mathfrak{C}|=|V(K)|$ and  a transversal $T$ of $\mathfrak{C}$ such that $K$ is
isomorphic to a spanning subgraph $H$ of $H(G,\mathfrak{C},T)$. A graph $G'$ can be obtained from $G$ by the disjoint union with a complete graph $K_S$ on vertex set $S:=V(K)-V(K')$. Let $\mathfrak{C}':=\mathfrak{C}\cup\{\{s\}\mid s \in S\}$ and $T':=T\cup S$. 
Then $K'$ is isomorphic to a spanning subgraph $H'$ of $H(G',\mathfrak{C}',T')$.
Since $K'$ has property (*), there is a rooted $H'$\!-certificate $c'=(V_t)_{t \in V(K')}$ in $G'$. Then $c:=(V_t)_{t \in V(K)}$ is a rooted $H$\!-certificate in $G$. 
If, conversely, every component of $K$ has property (*) and there is an arbitrary graph $G$ with a coloring $\mathfrak{C}$ with $|\mathfrak{C}|=|V(K)|$ and  a transversal $T$ of $\mathfrak{C}$ such that $K$ is
isomorphic to a spanning subgraph $H$ of $H(G,\mathfrak{C},T)$, then for each component $H_1,H_2,\dots$ of $H$ there are pairwise disjoint subgraphs $G_1,G_2,\dots$ of $G$ and $\mathfrak{C}_1,\mathfrak{C}_2,\dots$ and $T_1,T_2,\dots$ such that $H_i$ is a subgraph of $H(G_i,\mathfrak{C}_i,T_i)$ for $i\in\{1,2,\dots\}$. Since property (*) holds for $H_i$, $i\in\{1,2,\dots\}$, there is a rooted $H_i$-certificate $c_i$ in $G_i$. 
Then the union of $c_1,c_2,\dots$ (considered as subsets of $V(H_i)\times\mathfrak{P}(V(G))$) is a rooted $H$\!-certificate in $G$, so that $K$ has property (*).

Finally, assume $K'$ has property (*) and let $K$ be a subgraph of $K'$ with $|V(K)|<|V(K')|$. Let $L$ be the edgeless graph on $V(K')\setminus V(K)$. Then $K\cup L$ is a spanning subgraph of $K'$ and has property (*), and so has $K$ as one of its components. 
\end{proof}

Next, we show that a graph with property (*) keeps its property if we attach a pending edge. 

\begin{lemma}\label{Lpending}
Let $K$ be a graph and $q\in V(K)$ be a vertex of degree~$1$. If $K-q$ has property (*), then $K$ has property (*). 
\end{lemma}

\begin{proof}
Let $K$ and $q$ such that $K-q$ has property (*). Assume that $K$ does not have property (*).
Let $r$ be the neighbor of $q$ in $K$.
Then, there exists a graph $G$ with a coloring $\mathfrak{C}$ and a transversal $T$ of $\mathfrak{C}$
such that $K$ is isomorphic to a spanning subgraph $H$ of $H(G,\mathfrak{C},T)$ but $G$ has no rooted $H$\!-certificate.
We may take $G$ with $|V(G)|+|E(G)|$ minimal with respect to this property, implying that for all $A \not= B$ from $\mathfrak{C}$,
$G[A \cup B]$ has at most one single nontrivial component (note that there might be isolated vertices in $G[A \cup B]$) which induces a path between the vertices $a \in A \cap T$, $b \in B \cap T$ if $ab \in E(H)$
and $E(G[A \cup B])=\emptyset$ otherwise.
In particular, $H=H(G,\mathfrak{C},T)$. 
Now let $Q \not= R$ be the members of $\mathfrak{C}$ with $q \in Q$, $r \in R$. If there exists a vertex $x \in Q \setminus \{q\}$
then $x$ has degree $2$, and we take its neighbors $y,z \in R$, and contract $yxz$ to a single vertex $w$.
For $A \in \mathfrak{C}$ set $A':=(A \setminus \{y,z\}) \cup \{w\}$ if $A=R$, $A':=A \setminus \{x\}$ if $A=Q$, and $A':=A$ otherwise, and 
for $z \in T$ set $z':=w$ if $z \in \{y,z\}$ and $z':=z$ otherwise. Then $\mathfrak{C}':=\{A'\mid A \in \mathfrak{C}\}$ is a coloring of $G'$
and $T':=\{t'\mid t \in T\}$ is a transversal of $\mathfrak{C}'$. 

Next, we show that $H=H(G,\mathfrak{C},T)$ is isomorphic to $H(G',\mathfrak{C}',T')$ (via $t\mapsto t'$). 
Let $s\not= t\in T\setminus\{q,r\}$ with $st\in E(H)$, then neither of $x,y,z$ lays on a $2$-colored $s$,$t$-path $P_{st}$ in $G$, hence, $P_{st}$ is an $s'$,$t'$-path in $G'$. 
If $r\in\{y,z\}$ and $t\in T\setminus\{q,r\}$ such that $rt\in E(H)$, then  a $2$-colored $r$,$t$-path $P_{rt}$ in $G$ translates to a $r'$,$t'$-path if $\{y,z\}\nsubseteq V(P_{rt})$ or there exists a subpath of $P_{rt}$ that is an $r'$,$t'$-path if $\{y,z\}\subseteq V(P_{rt})$. 
If $r\notin\{y,z\}$ and $t\in T\setminus\{q,r\}$ such that $rt\in E(H)$, then  a $2$-colored $r$,$t$-path $P_{rt}$ in $G$ is a $r'$,$t'$-path by possibly replacing a vertex with $w$ if $\{y,z\}\nsubseteq V(P_{rt})$ or, if $\{y,z\}\subseteq V(P_{rt})$, there exist a subpath of $P_{rt}$ from $y$ to $z$ that can be replaced by $w$ to obtain an $r'$,$t'$-path. 
The $2$-colored $r$,$q$-path in $G$ contains $yxz$ as a subpath. By replacing this subpath with $w$, we obtain an $r'$,$q'$-path in $G'$. 
For any $s\not= t\in T$ such that  $st\notin E(H)$, the contraction of $yxz$ to a single vertex $w$ does not produce any new edges between two color classes where there haven't been any edges before. 
Hence, $H$ is isomorphic to $H(G',\mathfrak{C}',T')$.

By choice of $G$, $G'$ has a rooted $H(G',\mathfrak{C}',T')$-certificate. We get a rooted $H$\!-certificate
of $G$ by replacing its bag $B$ containing $w$ --- if any --- with $(B \setminus \{w\}) \cup \{y,x,z\}$, contradiction.
Therefore, $Q=\{q\}$ so that, by assumption, $G-q$ has a rooted $H(G-q,\mathfrak{C} \setminus \{Q\},T \setminus \{q\})$-certificate,
from which we get a rooted $H$\!-certificate by adding the bag $Q=\{q\}$.
\end{proof}

\section{Kempe chains and rooted \boldmath$K_7$-minors}
\label{S2}

If we like to find graphs that fail to have property (*), then we will preferably consider small graphs containing all the required paths between transversal vertices on one hand;
on the other hand, there should not be many edges incident with the transversal vertices. 
A construction that works well is the one presented in this section, which we call $Z(G )$ of a graph $G$. 
The construction is simple to describe; the resulting graph consists of $G$ plus a clone of each of its vertices.

For a graph $G$ we define the graph $Z(G)$ by $V(Z(G)):=V(G) \times \{1,2\}$ and
$E(Z(G)):=\{(x,i)(y,j)\mid xy \in E(G) \wedge (i \not=1 \vee j \not=1)\}$. That is, $Z(G)$ is obtained from the lexicographic product of $G$ with the
graph $(\{1,2\},\emptyset)$ by deleting all edges connecting vertices from $V(G) \times \{1\}$.
For $s=(x,i) \in V(Z(G))$ let us define $\overline{s}:=(x,3-i)$.
It follows that $\mathfrak{C}:=\{\{(x,1),(x,2)\}\mid x \in V(G)\}=\{\{s,\overline{s}\}\mid s \in V(Z(G))\}$ is a coloring of $Z(G)$
--- the {\em canonical coloring} --- 
and that $T:=V(G) \times \{1\}$ is a transversal of $\mathfrak{C}$. Observe that $T$ induces an anticlique in $Z(G)$.
As the union of two color classes $\{(x,1),(x,2)\},\{(y,1),(y,2)\}$ induce \textbf{(i)} a path of length three between its transversal vertices if $xy \in E(G)$ and
\textbf{(ii)} an edgeless graph if $xy \not\in E(G)$, we see that $H:=H(Z(G),\mathfrak{C},T)$ is isomorphic to $G$ (via $(x,1) \mapsto x$).
Moreover, we find a copy of $G$ induced in $Z(G)$ in a very natural way: $Z(G)[V(G) \times \{2\}]$ is isomorphic to $G$ (via $(x,2) \mapsto x$).
We study whether there exists a rooted $H$\!-certificate in $Z(G)$ for different graphs $G$.

The bags of any $H$\!-certificate $c=(V_t)_{t \in T}$ in $Z(G)$
have average order at most $2$. That is, as soon as there are bags of order at least $3$,
there must be bags of order $1$; locally, the inverse implication is almost true, as follows:

\begin{cl}\label{claim1}
If $st \in E(H)$ is not on any triangle of $H$, then $|V_s|=1$ implies $|V_t|\geq 3$.
\end{cl}

\begin{smallproof}
Suppose that $|V_s|=1$, that is, $V_s=\{s\}$. Since $s,t$ are not adjacent in $Z(G)$, $|V_t| \geq 2$. If $|V_t|=2$, then
$V_t=\{t,u\}$ for some $u \in V(Z(G))$, where $t,u$ and $s,u$ are adjacent in $Z(G)$ so that  $u \in V(Z(G)) \setminus V(H)$, $\overline{u}\neq s$ and
$t,\overline{u}$ must be adjacent in $H$.
Since $s,t,\overline{u}$ do not form a triangle in $H$,
$s,\overline{u}$ are nonadjacent in $H$ so that $s,u$ are nonadjacent in $Z(G)$; consequently, $s$ has no neighbors in $V_t$, contradiction.
This implies $|V_t| \geq 3$ as claimed.
\end{smallproof}

If all bags of $c$ have order $2$, then we look at the function $f\colon V(G) \to V(G)$ defined by $f(x):=y$ if $V_{(x,1)}=\{(x,1),(y,2)\}$.
Since the bags are disjoint, $f$ is an injection and, thus, a permutation of $V(G)$.
Since the bags are connected, $xf(x) \in E(G)$, so that we may represent $f$ as a partial orientation of $G$, where
$xy$ is oriented from $x$ to $y$ if $y=f(x)$ and from $y$ to $x$ if $x=f(y)$ (which may happen simultaneously).
As $c$ is a rooted $H$\!-certificate in $Z(G)$ we know that $xy \in E(G)$ implies that $V_{(x,1)}, V_{(y,1)}$ are adjacent,
which is equivalent to saying that, in $Z(G)$, $f(x)$ is adjacent to one of $y,f(y)$ or $f(y)$ is adjacent to one of $x,f(x)$.
Conversely, if $f$ is a permutation of $V(G)$ such that \textbf{(i)} $xf(x) \in E(G)$ for all $x\in V(G)$ and \textbf{(ii)} $xy \in E(G)$ implies that
$f(x)$ is adjacent to one of $y,f(y)$ or $f(y)$ is adjacent to one of $x,f(x)$, then $V_{(x,1)}:=\{(x,1),(f(x),2)\}$ defines 
an $H$\!-certificate in $Z(G)$. 
Let us call a permutation of $V(G)$ with \textbf{(i)} and \textbf{(ii)} a {\em good} permutation throughout this section.

\begin{cl}\label{claim2}
If $G$ has a good permutation, then every vertex of degree at least $3$ in $G$ is on a cycle of length at most $4$ in $G$.
\end{cl}

\begin{smallproof}
Let $f$ be a good permutation and suppose that $w$ is a vertex of degree at least $3$ in $G$ and let $x,y,z$ be three neighbors of $w$ in $G$, where $f(w)=x$.
We may assume that $f(y) \not= w$ (otherwise $f(z) \not= w$ and we swap the roles of $y,z$). If $u:=f(y)\neq w$ is a neighbor of $w$
then by \textbf{(i)}, $w,y,u$ form a triangle and we are done. Otherwise, $\{w,f(w)=x\}$, $\{y,f(y)=u\}$ are disjoint, and \textbf{(ii)} implies
that, in $G$, $u$ is a neighbor of $w$ or $x$, or that $x$ is a neighbor of $y$ or $u$; in either case, $w$ is on a cycle of length $3$ or $4$.
This proves Claim~\ref{claim2}. 
\end{smallproof}

Let us specialize the considerations to the graph $G$ obtained from a cycle $G'$ of length $6$ by adding another vertex $x$ and two edges
connecting $x$ to two vertices $a,b$ at distance $3$ on $G'$.
Assume, to the contrary, that $Z(G)$ has an $H$\!-certificate $(V_t)_{t \in T}$ with $T=V(G) \times \{1\}$.
Let $A$ be the set of vertices $t \in T$ with $|V_t|=1$.
By Claim~\ref{claim2}, $G$ cannot have a good permutation, so $|A| \geq 1$. 
$A$ is an anticlique in $H$ (and in $Z(G)$), so $|A| \leq 3$, and, by Claim~1, $|V_s| \geq 3$ for every vertex $s$ in the neighborhood of $A$ in $H$.
For each case $|A|=1,|A|=2,|A|=3$ one readily verifies $|N_H(A)| \geq |A|+1$.
It follows $q:=\sum_{t \in T} |V_t| \geq 3 \cdot (|A|+1)+2 \cdot (7-2|A|-1)+1 \cdot |A|=15$,
contradicting $q \leq |V(Z(G))|=14$. It follows that $G$ does not have property (*).
By Theorem~\ref{Tsubgraphs} (property (*) inherits to spanning subgraphs) we conclude that $K_7$ does not have it either.
For a constructive example, we could take $Z(G)$ with $\mathfrak{C}$ and $T$ as above and just add all edges
between transversal vertices $(x,1),(y,1)$ with $xy \not\in E(G)$ to obtain a graph $G'$ without a rooted $H(G',\mathfrak{C},T)$-certificate,
where $H(G',\mathfrak{C},T)$ is now the complete graph on the seven vertices from $T$. So we have proved:

\begin{theorem}
  \label{T1}
  $K_7$ does not have property (*).
\end{theorem}

Let $d \geq 3$. We now specialize to a connected, $d$-regular, nonbipartite graph $G$ of girth at least $5$
and assume, to the contrary, that $Z(G)$ has an $H$\!-certificate $c=(V_t)_{t \in T}$.
Let $A,B,C$ be the set of vertices $x \in V(G)$ with $|V_{(x,1)}|$ being $1$, $2$, and at least $3$, respectively.
By Claim~\ref{claim2}, $G$ does not have a good permutation, so that $|A| \geq 1$.
The vertices from $A$ and $A\times\{1\}$ induce an edgeless graph in $G$ and $Z(G)$, respectively, and the neighbors of $A$ in $G$ are all from $C$ by Claim~\ref{claim1}.
The number of edges between $A$ and $C$ in $G$ is thus equal to $d|A|$ and at most $d|C|$ with equality only if every vertex from $C$ has
all its $d$ neighbors in $A$. However, in the latter case, $G[A \cup C]$ is $d$-regular and bipartite and $B$ is empty as $G$ is connected,
so that $G$ is bipartite, contradiction. It follows $d|A|<d|C|$ and, consequently,
$|Z(G)| \geq \sum_{t \in T} |V_t| \geq |A|+2|B|+3|C|>2|A|+2|B|+2|C|=|Z(G)|$, contradiction. So we have proved

\need{3cm}
\begin{theorem}
  \label{T2}
  If $G$ is connected, $d$-regular with $d \geq 3$, nonbipartite of girth at least $5$, then it does not have property (*).
\end{theorem}

The smallest graph meeting the assumptions of Theorem~\ref{T2} is, incidentally, the Petersen graph.

\section{Unicyclically arranged Kempe chains}
\label{S3}

We continue with a number of positive results. 
The main result of Fabila-Monroy and Wood in~\cite[Theorem~8]{FabilaMonroyWood2013} states that for any 3-connected graph $G$ and distinct vertices $t_1,t_2,t_3,t_4\in V(G)$
such that two vertex disjoint $t_i$,$t_j$-path and $t_k$,$t_l$-path exist for each choice of distinct $i,j,k,l\in\{1,2,3,4\}$,
then there exists a rooted $H$\!-certificate where $H$ is the complete graph on $\{t_1,\dots,t_4\}$. 
This generalizes to:
\begin{theorem}
  \label{T3}
  Every graph on at most four vertices has property (*).
\end{theorem}

\begin{proof}
By Theorem~\ref{Tsubgraphs}, it suffices to prove that $K_4$ has property (*). 
Let $G$ be a graph, $\mathfrak{C}$ be a coloring of $G$ with $|\mathfrak{C}|=4$, and let $T$ be a transversal of $\mathfrak{C}$.
For $x \in V(G)$, let $A_x$ denote the member of $\mathfrak{C}$ containing $x$, and
suppose that for all $x \neq y$ from $T$ there exists an $x$,$y$-path $P_{xy}$ in $G[A_x \cup A_y]$,
that is, $H:=H(G,\mathfrak{C},T)$ is a complete graph on four vertices.
Choose $G$ to be a counterexample with $|V(G)|+|E(G)|$ minimal. Then $G$ is connected and $E(G)=\bigcup_{x\neq y}E(P_{xy})$.
By the previously stated result of Fabila-Monroy and Wood, $G$ is not 3-connected. 

We may assume that $G$ has a separator $S$ with $|S|\leq 2$. 
If $S \subseteq A_x$ for some $x \in T$, then $T \setminus \{x\}$ is contained in some component $C$ of $G-S$.
Let $G'$ be the graph obtained from $G$ by contracting $X:=V(G) \setminus V(C)$ to a single vertex $w$.
For $A \in \mathfrak{C}$ set $A':=(A \setminus X) \cup \{w\}$ if $A=A_x$ and $A':=A \setminus X$ otherwise, and 
for $z \in T$ set $z':=w$ if $z \in X$ and $z':=z$ otherwise. Then $\mathfrak{C}':=\{A'\mid A \in \mathfrak{C}\}$ is a coloring of $G'$
and $T':=\{t'\mid t \in T\}$ is a transversal of $\mathfrak{C}'$. 
For all $u' \neq y'$ from $T\setminus\{x'\}$ there exists the same $u'$,$y'$-path $P_{u'y'}$ in $G'[A'_{u'} \cup A'_{y'}]$ as in $G$.
If $x'\neq w$, then for all $y'\in T\setminus\{x'\}$, either an $x'$,$y'$-path $P_{x'y'}$ is still contained in $C$ 
or there exist an $x'$,$S$-path $P_1$ and a $y'$,$S$-path $P_2$ in $G[A_x \cup A_y]$, which can be concatenated to an $x'$,$y'$-path $P_{x'y'}$ in $G'[A'_{x'} \cup A'_{y'}]$ in a natural way. 
If $x'=w$, then for all $y'\in T\setminus\{x'\}$ a shortest $y'$,$S$-path $P_1$ in $G[A_x \cup A_y]$ (exists because of an  $x$,$y$-path $P_{xy}$ in $G[A_x \cup A_y]$) translates straightforward to an  $x'$,$y'$-path $P_{x'y'}$ in $G'[A'_{x'} \cup A'_{y'}]$.
Thus, $H':=H(G,\mathfrak{C}',T')$ is a complete graph on $T'$.
By the choice of $G$, there exists a rooted $H'$\!-certificate in $G'$, that can be extended to a rooted $H$\!-certificate
of $G$ by replacing its bag $B$ containing $w$ --- if any --- with $(B \setminus \{w\}) \cup X$, contradiction.

Otherwise, $G$ is 2-connected and there exist $x \neq y$ from $T$
such that each of $S \cap A_x$ and $S \cap A_y$ consists of a single vertex $x_0$, $y_0$, respectively.
Again, $T \setminus \{x,y\}$ is contained in the same component $C$ of $G-S$.
Let $G'$ be the graph obtained from $G$ by deleting $X:=V(G) \setminus (V(C) \cup S)$ and adding an edge $x_0y_0$ (if it does not already exist).
For $A \in \mathfrak{C}$, let $A':=A \setminus X$, so that $\mathfrak{C}':=\{A'\mid A \in \mathfrak{C}\}$ is a coloring of $G'$.
For $z \in T$, set $z':=z_0$ if $z \in \{x,y\} \cap X$ and $z':=z$ otherwise, so that $T':=\{z'\mid z \in T\}$ is a transversal of $\mathfrak{C}'$.
If $x'\neq x_0$, then for all $u'\in T\setminus\{x',y'\}$, either an $x'$,$u'$-path $P_{x'u'}$ is still contained in $C$ 
or there exist an $x'$,$x_0$-path $P_1$ and a $u'$,$x_0$-path $P_2$ in $G[A_x \cup A_y]$, which can be glued together to an $x'$,$u'$-path $P_{x'u'}$ in $G'[A'_{x'} \cup A'_{u'}]$.
If $x'=x_0$, then for all $u'\in T\setminus\{x',y'\}$ a shortest $u'$,$x_0$-path $P_1$ in $G[A_x \cup A_u]$ (exists because of an  $x$,$u$-path $P_{xu}$ in $G[A_x \cup A_u]$) is an  $x'$,$u'$-path $P_{x'u'}$ in $G'[A'_{x'} \cup A'_{u'}]$.
The existence of a $y'$,$u'$-path for all $u'\in T\setminus\{x',y'\}$ holds analogously and the same arguments verifies the existence of an  $x'$,$y'$-path $P_{x'y'}$ in $G'[A'_{x'} \cup A'_{y'}]$ using the edge $x_0y_0$ if necessary. 
Thus, $H':=H(G',\mathfrak{C}',T)$ is a complete graph on $T'$.
By the choice of $G$, $G'$ admits a rooted $H'$\!-certificate $c'$. 
If $X$ does not contain both $x$ and $y$, say, $y \notin X$, then 
$c'$ can be extended to a rooted $H$\!-certificate of $G$ by replacing its bag $B$ containing $x_0$ --- if any --- with $B \cup X$, contradiction.
If, otherwise, $X$ contains both $x$ and $y$, then there are two vertex disjoint paths $P_1$ and $P_2$ connecting $x_0$ and $y_0$ to $\{x,y\}$, respectively (by 2-connectivity of $G$ and Menger's Theorem). 
It is obvious that $V(P_1),V(P_2)\subseteq S\cup X$. 
Let $X_0:=V(P_1)$ and let $Y_0$ be the vertex set
of the component of $G[S \cup X]- X_0$ containing $P_2$, and therefore also $y_0$ and $\{x,y\}\setminus X_0$. 
Since $G$ is $2$-connected and all neighbors of $Y_0 \setminus \{y_0\} \neq \emptyset$
are in $X_0 \cup \{y_0\}$, $X_0$ and $Y_0$ are adjacent in $G$. 
Now, a rooted $H$\!-certificate can be obtained from $c'$ by replacing the bags $B_1$ and $B_2$ containing $x_0$ and $y_0$ with 
$B_1\cup X_0$ and $B_2\cup Y_0$, respectively.
Thus, all bags contain exactly one vertex from $T$ and are pairwise adjacent, contradiction. 
\end{proof}

Here is an infinite class of connected graphs for which (*) is true.

\begin{lemma}
  \label{L1}
  Every cycle has property (*).
\end{lemma}

\begin{proof}
Given $\ell$, we have to prove for every graph $G$, every coloring $\mathfrak{C}=\{A_0,\dots,A_{\ell-1}\}$ and every choice
$t_i \in A_i$ for $i \in \{0,\dots,\ell-1\}$ such that there exists a $t_i$,$t_{i+1}$-path $P_i$ in $G[A_i \cup A_{i+1}]$ for all $i \in \{0,\dots,\ell-1\}$,
indices modulo $\ell$, there exists a rooted $H:=(\{t_0,\dots,t_{\ell-1}\},\{t_0t_1,t_1t_2,\dots,t_{\ell-2}t_{\ell-1},t_{\ell-1}t_0\})$-certificate.
Suppose that $G$ was a counterexample with $|V(G)|+|E(G)|$ minimal. Then $E(G)=E(P_0) \cup \dots \cup E(P_{\ell-1})$.
It follows that very vertex has degree at least $2$.
If $x \in A_i \setminus \{t_i\}$ had only two neighbors, say, $y,z$ in $A_j$, where $j \in \{i+1,i-1\}$, then let $G'$ be obtained from $G$ by contracting $yxz$
to a new vertex $w$, define $A_i':=A_i \setminus \{x\}$, $A_j':=(A_j \setminus \{y,z\}) \cup \{w\}$, $A_k':=A_k$ for $k \in \{0,\dots,\ell-1\} \setminus \{i,j\}$, $t_k':=w$ if $k=j \wedge t_k \in \{y,z\}$, and $t_k':=t_k$ otherwise. 
Letting $H':=(\{t'_0,\dots,t'_{\ell-1}\},\{t'_0t'_1,t'_1t'_2,\dots,t'_{\ell-2}t'_{\ell-1},t'_{\ell-1}t'_0\})$,
we know that, by the choice of $G$, there exists a rooted $H'$\!-certificate in $G'$, from which we can construct a rooted $H$\!-certificate
of $G$ by replacing its bag $B$ containing $w$ --- if any --- with $(B \setminus \{w\}) \cup \{y,x,z\}$, contradiction.
Hence we may assume that every vertex in $A_i \setminus \{t_i\}$ has degree $4$, that is, it is on both $P_i$ and $P_{i-1}$.
In particular, all $A_i$ have the same order $d \geq 1$. If $d=1$, then $G=H$, so that $G$ is not a counterexample, contradiction.
Hence $d \geq 2$, and we consider $G':=G - \{t_0,\dots,t_{\ell-1}\}$, $A_i':=A_i \setminus \{t_i\}$ and let $t'_i \in A_i'$ be the
neighbor of $t_{i-1}$ on $P_{i-1}$. As $t'_{i+1}$ is on $P_i - \{t_i,t_{i+1}\}=G[A'_i \cup A'_{i+1}]$, we know by choice of $G$
that $G'$ has a rooted $H'$\!-certificate ($H'$ as above), from which we get a rooted $H$\!-certificate of $G$ by extending the bag
containing $t'_i$ by the vertex $t_{i-1}$, contradiction.
\end{proof}

Lemma~\ref{L1} generalizes to unicylic graphs, as follows.

\begin{theorem} 
  \label{T4}
  Every (connected) graph with at most one cycle has property (*).
\end{theorem}

\begin{proof}
Let $K$ be a connected graph with at most one cycle. 
We proceed by induction on the number of edges not laying on a cycle of $K$. 
By Lemma~\ref{L1} and the fact that property (*) holds for $K_1$, we conclude that $K$ has property (*) if there is no edge not laying on a cycle. 
Otherwise let $q\in V(K)$ be a vertex of $K$ of degree $1$. By induction $K-q$ has property (*) and so has $K$ by Lemma~\ref{Lpending}. 
\end{proof}

\section{The graphs \boldmath$Z(G)$ for \boldmath$|V(G)| \leq 6$}

One could ask if Theorem~\ref{T1} extends to smaller complete graphs or, alternatively, if the bound (``four'')
in Theorem~\ref{T3} can be increased. Both questions need new methods:
In this section, we will see that the method used in Section~\ref{S2} to identify graphs not satisfying (*) 
does not work for graphs on less than seven vertices, whereas, in the next section,
we will collect our knowledge on (*) for graphs on five vertices.

Let $H$ be a graph and $A\subseteq V(H)$ an anticlique in $H$. 
Let us call $A$ {\em matchable} if there exists a matching $M$ from $V(H) \setminus A$ into $A$ (covering $V(H) \setminus A$, but not necessarily covering $A$).
We start with another positive result if a graph contains a matchable anticlique. 

\begin{lemma}
  \label{L2}
  Let $G$ be a graph, $\mathfrak{C}$ be a coloring of $G$, let $T$ be a transversal of $\mathfrak{C}$ and $H:=H(G,\mathfrak{C},T)$.
  Assume that $A$ is a matchable anticlique in $H$ and $M$ a matching covering $V(H) \setminus A$. 

  Suppose that for every edge $st \in M$ where $s \in V(H) \setminus A$ and $t \in A$, it holds that
  $G[P \cup Q]-t$ is connected, where $P,Q$ denotes the color class from $\mathfrak{C}$ containing $s,t$, respectively.
  Then there is a rooted $H$\!-certificate in $G$.
\end{lemma}

\begin{proof}
For $s \in A$ set $V_s:=\{s\}$. For $s \in V(H) \setminus A$ let $t$ be the vertex in $A$ such that $st \in M$,
let $P,Q$ be the classes of $\mathfrak{C}$ containing $s,t$, respectively, and set $V_s:=(P \cup Q) \setminus \{t\}$.
One readily verifies that $(V_s)_{s \in V(H)}$ is a rooted $H$\!-certificate.
\end{proof}

\begin{theorem}
  Let $G$ be any graph with at most $6$ vertices. Consider $Z(G)$ defined as in Section~\ref{S2} with its canonical coloring
  $\mathfrak{C}=\{\{x\} \times \{1,2\}\mid x \in V(G)\}$ and the transversal $T:=V(G) \times \{1\}$.
  Then $Z(G)$ has a rooted $H(Z(G),\mathfrak{C},T)$-certificate.
\end{theorem}

\begin{proof}
Let $G$ be a counterexample and set $H:=H(Z(G),\mathfrak{C},T)$. 
Note that $G$ is isomorphic to $H$. 
By the positive results of the Section~\ref{S3}, 
we may assume that $|V(H)| \in \{5,6\}$, that $H$ is connected, and that $H$ contains more than one cycle.
If $H$ has  a cutvertex $s$, then there is a component $C$ of $H-s$ with one or two vertices;
if $V(C)=\{t\}$, then we know that $Z(G)-\{s,\overline{s},t,\overline{t}\}$ has a rooted $(H-\{s,t\})$-certificate,
and we extend it to a rooted $H$\!-certificate of $Z(G)$ by setting $V_s:=\{s,\overline{s},\overline{t}\}$ and $V_t:=\{t\}$;
if, otherwise, $V(C)=\{t,u\}$, then $Z(G)-\{t,\overline{t},u,\overline{u}\}$ has a rooted $(H-\{t,u\})$-certificate,
and we extend it to a rooted $H$\!-certificate of $Z(G)$ by setting $V_t:=\{t,\overline{u}\}$ and $V_u:=\{u,\overline{t}\}$.
Therefore, we may assume that $H$ is $2$-connected.

By Lemma~\ref{L2}, we may assume that $H$ has no matchable anticlique.
For $|V(H)|=5$ it follows that there is no anticlique $A$ of order larger than $2$ in $H$, as it would be matchable by $2$-connectivity of $H$.
Consequently, $H$ has a spanning $5$-cycle $t_0,t_1,t_2,t_3,t_4$, and setting $V_{t_i}:=\{t_i,\overline{t_{i+1}}\}$ (indices mod 5) yields a family
of pairwise adjacent cliques and, thus, a rooted $H$\!-certificate.

For $|V(H)|=6$ it follows that there is no anticlique $A$ of order larger than $3$ (as, again, $A$ would be matchable).
If $A$ was an anticlique of order three in $H$, then there would be a matching $M$ from $A$ to $V(H) \setminus A$ by Hall's Theorem
(see~\cite{Diestel2017}), since every vertex in $A$
has two neighbors in $V(H) \setminus A$ and $N_H(A)=V(H) \setminus A$ as there are no anticliques of order $4$.
$M$ is a matching from $V(H) \setminus A$ into $A$, too, so $A$ would be matchable, contradiction. It follows
that $H$ has no anticliques of order larger than $2$.

Let us say that a matching $N=\{r_1s_1,r_2s_2,r_3s_3\}$ of $H$ is {\em good} if every edge from $E(H) \setminus N$
is either on a triangle containing one edge from $N$ or on a cycle of length $4$ containing two edges from $N$.
If there is such a good matching, $V_{r_i}:=\{r_i,\overline{s_i}\}$, $V_{s_i}:=\{s_i,\overline{r_i}\}$ for $i \in \{1,2,3\}$ defines a rooted $H$\!-certificate.

If $H$ has a spanning cycle $t_0t_1t_2t_3t_4t_5$, then setting $V_{t_i}:=\{t_i,\overline{t_{i+1}}\}$ yields a family
of cliques such that $V_{t_i}$ is adjacent to $V_{t_{i+1}}$ and $V_{t_{i+2}}$ ((sub-) indices modulo $6$). So we get a rooted $H$\!-certificate
in the case that all three {\em long chords} $t_0t_3,t_1t_4,t_2t_5$ of the cycle are missing.
If all long chords are present, then they form a good matching in $H$, and we are done, too.

Suppose that $S$ is a smallest separator in $H$ and that $C,D$ are two components of $H-S$. Then $|S| \geq 2$, $C,D$ are complete, and
there are no further components of $H-S$ as $H$ has no anticlique of order three. If $|V(C)|=|V(D)|=|S|=2$, then
there is a spanning cycle $t_0t_1t_2t_3t_4t_5$ in $H$ with $V(C)=\{t_0,t_1\}$, $S=\{t_2,t_5\}$, and $V(D)=\{t_3,t_4\}$;
if $t_2t_5 \not \in E(H)$, then all long chords are missing, and we are done, otherwise $\{t_0t_1,t_2t_5,t_3t_4\}$ is a good matching.

Suppose that $V(C)=\{s\},S=\{t,u\},V(D)=\{a,b,c\}$. We may assume that both $t$ and $u$ have more than one neighbor among $a,b,c$,
for otherwise $N_H(\{s,t\})$ or $N_H(\{s,u\})$ would be a separator of order $2$ as discussed in the previous paragraph.
If $N_H(t) \setminus \{u\} =N_H(u) \setminus \{t\} =\{s,a,b\}$, then $\{st,ua,bc\}$ is a good matching, and we are done.
Otherwise, $ta,tb,ub,uc \in E(H)$ without loss of generality, and if $tc \not\in E(H)$ and $ua \not\in E(H)$, then 
$stabcu$ is a cycle of length $6$ without long chords, and we are done. So $tc \in E(H)$ without loss of generality.
If $ua \in E(H)$, too, then $su,ta,bc$ is a good matching, otherwise $V_s:=\{s,\overline{u}\}$, $V_u:=\{u,\overline{s}\}$,
$V_t:=\{t,\overline{c}\}$, $V_c:=\{c,\overline{b}\}$, $V_b:=\{b,\overline{t}\}$, $V_a:=\{a\}$ defines a rooted $H$\!-certificate.

From now on we may assume that $H$ is $3$-connected. If $H$ has a spanning wheel with center $s$ and rim cycle
$t_0,\dots,t_4$, then $V_s:=\{s\}$, $V_{t_i}:=\{t_i,\overline{t_{i+1}}\}$ (indices mod 5) defines a rooted $K_T$\!-certificate.
If $H$ has a spanning prism with triangles $s_0s_1s_2$, $t_0t_1t_2$ and connecting edges $s_it_i$, then
the connecting edges form a good matching, which in this case define already a rooted $K_T$\!-certificate.

To conclude the proof, we will now show that $H$ either has a spanning wheel or a spanning prism.
If $|S|=3$, then $V(C)=\{s\}$, $S=\{t,u,v\}$, $V(D)=\{a,b\}$,
$ta,ua,ub,vb \in E(H)$ without loss of generality. If $tv \in E(H)$, then we have a spanning prism with triangles $stv$ and $uab$. 
Otherwise, one of $tu,uv$ is in $E(H)$ (as $S$ is not an anticlique), say $tu \in E(H)$. If $uv \in E(H)$, then
we have a spanning wheel with center $u$, otherwise $va \in E(H)$ and we have a spanning prism with triangles $stu$ and $vab$.

Hence we may assume that $H$ is $4$-connected and, therefore, contains $C_6^2$.
Consequently, it has a spanning prism, and we are done.
\end{proof}

\section{Connected transversals of \boldmath$5$-colorings}

By Theorem~\ref{T3}, all graphs with at most four vertices have property (*), 
whereas by Theorem~\ref{T1}, there exists a graph on seven vertices which does not have property (*).
For graphs $K$ on five vertices we do not have the full picture; since we may assume that such a $K$ is connected and since
a connected graph on five vertices and at most five edges contains at most one cycle, we know by Theorem~\ref{T4} that
all graphs on five vertices and at most five edges have property (*), too. This extends as follows:

\begin{theorem}
  \label{T6}
  Every graph on five vertices and at most six edges has property (*).
\end{theorem}

\begin{proof}
Let $K$ be a graph with $|V(K)|=5$ and $|E(K)| \leq 6$. 
If $K$ is not connected, then $K$ has property (*) by Theorems~\ref{Tsubgraphs} and~\ref{T3}. 
If $K$ is connected and $|E(K)|<6$, then the theorem follows from Theorem~\ref{T4}. 
Therefore, 
we may assume that $K$ is connected and $|E(K)|=6$.
If there is a vertex $q$ of degree $1$ in $K$, then $K-q$ and $K$ have property (*) by Theorem~\ref{T3} and Lemma~\ref{Lpending}, respectively. 
Thus, up to isomorphism, there are only three remaining graphs to consider:
The \emph{hourglass $\hourclass$} obtained from the union of two disjoint triangles  by identifying two nonadjacent vertices, the complete bipartite
graph $K_{2,3}$ with color classes of order $2$ and $3$, respectively, and the graph $C_5^+$ obtained from a $5$-cycle by adding a edge
connecting some pair of nonadjacent vertices.

Assume, to the contrary, that $K$ does not have property (*); then there exists a graph $G$ with a coloring $\mathfrak{C}$ and a transversal $T$ of $\mathfrak{C}$ such that $K$ is isomorphic to a spanning subgraph $H$ of $H(G,\mathfrak{C}, T)$ but $G$ has no rooted $H$\!-certificate. 
Again we may take $G$ with minimal $|V(G)|+|E(G)|$ with respect to this property, implying that for all $A \neq B$ from $\mathfrak{C}$, $G[A \cup B]$ has a single nontrivial component which induces a path between the unique vertices $a\in A\cap T$, $b\in B\cap T$ if $ab\in E(H)$ and $E(G[A \cup B]) = \emptyset$ otherwise.
In particular, $H(G,\mathfrak{C}, T)$ is isomorphic to $K$.
As in the proofs of Lemma~\ref{Lpending} and~\ref{L1}, we may assume that all vertices in $V(G)\setminus T$ have degree at least $4$.

In all three cases, let $T:=\{t_1,\dots,t_5\}$ and $\mathfrak{C}:=\{A_1,\dots,A_5\}$ such that $t_i \in A_i$ for all $i \in \{1,\dots,5\}$;
the $t_i$ will be specified differently in each case. In each case, we will find a rooted $H$\!-certificate $c$ defined by its bags $B_i=:c(t_i)$

\textbf{Case~1.} $K$ is isomorphic to the hourglass $\hourclass$.

Let $t_1 \in T$ be the vertex of degree $4$ in $H$, and let
 $s_2, s_3$ be two neighbors in $G$ of $t_1$ in the color classes $A_2$ and $A_3$, respectively, such that $t_2t_3\in E(H)$.
Then there is a  $t_2$,$t_3$-path $P$ in $G[A_2 \cup A_3]$, and, because of the assumptions on $G$, $s_2,s_3\in A_2\cup A_3=V(P)$.
It is possible to partition $V(P)$ into two bags $B_2$ and $B_3$ such that each of them contains exactly
one vertex from $\{s_2, s_3\}$ and one from  $\{t_2, t_3\}$, and $G[B_2],G[B_3]$ are connected subgraphs. 
Repeating this step for the other two neighbors of $t_1$ in $G$, we obtain bags $B_2,\dots ,B_5$ 
forming a rooted $H$\!-certificate together with the fifth bag $B_1:=\{t_1\}$, contradiction.

\textbf{Case~2.} $K$ is isomorphic to the graph $K_{2,3}$.

First, note that all Kempe chains in $G$ have at least four vertices.
(Otherwise remove one edge connecting two transversal vertices; the remaining graph is unicyclic and we are done by Theorem~\ref{T4}.)
Additionally, assume $|V(G)|$ to be minimal. 

Let $t_1,t_2,t_3$ be the vertices of $T$ of degree $2$ in $H$ and let $s_i$ (not necessarily distinct) be a neighbor of $t_i$ for $i\in\{1,2 ,3\}$
such that $s_1,s_2 ,s_3\in A_4$. 
Each of the vertices $s_i$ has at least two neighbors in a color class other than $A_i$.  
Assume first that two among $s_1,s_2 ,s_3$ have such neighbors in a common color class (this will always happens if $s_1,s_2 ,s_3$ are not pairwise distinct);
say, without loss of generality, there is at least one neighbor of $s_1$ and $s_2$ in $A_3$.
We set $B_1:=\{t_1\},B_2:=\{t_2\},B_3:=\{t_3\},B_4:=V(P_{34})\setminus\{t_3\}, B_5:=(V(P_{15})\setminus\{t_1\})\cup(V(P_{25})\setminus\{t_2\})$,
where $P_{ij}$ is the path from $t_i$ to $t_j$ in $G[A_i\cup A_j]$. 
Because each vertex in $A_5$  is a vertex of $P_{15}$ or $P_{25}$
--- all these vertices have degree at least $3$ --- there are edges between $B_5$ and $B_1,B_2 ,B_3$. 
Since $s_1$ and $s_2$ have a neighbor in $A_3$, we conclude $\{s_1,s_2\}\subseteq V(P_{34})$, and
$B_1,\dots ,B_5$ form a rooted $H$\!-certificate of $G$, contradiction. 

Thus, for $i\in \{1,2,3\}$, the vertices  $s_i$ are distinct and each has  a neighbor in a color class $\tilde{A_i}\neq A_i$ such that $\tilde{A_1},\tilde{A_2},\tilde{A_3}$ are distinct.
Then $s_i$ is a vertex of the 2-colored path from $t_4$ to the transversal vertex of $\tilde{A_i}$. 
Let $u_i$ be the neighbor of $s_i$ in $\tilde{A_i}$ with shortest distance to $t_4$ on this 2-colored path. 
Moreover, $u_1,u_2 ,u_3$ are colored differently as $\tilde{A_1},\tilde{A_2},\tilde{A_3}$ are distinct. 
Since $u_i\notin \{t_1,t_2,t_3\}$, all $s_i, t_i, u_i$ ($i \in \{1,2,3\}$) are distinct. 
Consider the graph $G':=G-\{s_1, t_1,s_2, t_2,s_3, t_3\}$ with the induced coloring $\mathfrak{C}':=\{A \cap V(G')\mid A \in \mathfrak{C}\}$, 
and $T':=\{u_1,u_2,u_3,t_4,t_5\}$. 
All vertices in $V(G)\setminus T$ have degree at least $4$ in $G$, thus, $u_i$ has neighbors in $A_5$
and it is on the $2$-colored path from the transversal vertex of  $\tilde{A_i}$ to $t_5$. 
Because of the choice of $u_1,u_2 ,u_3$, there is a  2-colored path from $u_i$ to $t_4$ for $i\in \{1,2,3\}$ in $G'$.
Thus, $H(G',\mathfrak{C}', T')$ has a spanning subgraph $H'$ isomorphic to $K$. 
Because of the minimality of $G$, there is a rooted $H'$\!-certificate in $G'$. 
Adding to its bag containing $u_i$ the vertices $s_i, t_i$ for $i\in \{1,2,3\}$, we obtain a rooted $H$\!-certificate of $G$, contradiction.

\textbf{Case~3.} $K$ is isomorphic to the graph $C_5^+$.

Let $t_1 \in T$ be the vertex of degree $2$ in $H$ in the unique triangle of $H$,
let $t_2, t_3 \in T$ be the two vertices of degree $3$ in $H$, and let $t_4, t_5$ the remaining two transversal vertices such that $t_2t_4\in E(H)$. 
Choose an arbitrary partition of $A_1\cup A_2\cup A_3$ into $B_1,B_2,B_3$ such that $B_1=\{t_1\}$ and $t_2\in B_2, t_3\in B_3$, $G[B_2],G[B_3]$ are connected subgraphs, and $B_1,B_2 ,B_3$ are bags of a rooted $K_S$-certificate with $S:=\{t_1,t_2,t_3\}$. 
Note that $t_1$ has two neighbors on the  $t_2$,$t_3$-path in $G[A_2\cup A_3]$.
If $t_4$ has a neighbor in $B_2$, then set $B_4:=\{t_4\}$ and $B_5:=(A_4\cup A_5)\setminus \{t_4\}$ as to obtain a rooted $H$\!-certificate,
contradiction. 
By symmetry, $t_5$ has no neighbor in $B_3$. 
But then, consider the $t_4$,$t_2$-path $P$ in $G[A_2\cup A_4]$.
This path starts with $t_4$ followed by a vertex in $B_3$ and ends in $t_2\in B_2$. Thus, there is a vertex $v\in A_4$ having neighbors in both $B_2$ and $B_3$. Since there is a $t_5$,$t_3$-path $Q$ in $G[A_3 \cup A_5]$, disjoint from $P$,
there is another vertex $w\in A_5$ having neighbors in both $B_2$ and $B_3$. 
Due to the assumptions to $G$, $v$ and $w$ have degree 4 and, therefore, they are vertices on the $t_4$,$t_5$-path $S$ in $G[A_4 \cup A_5]$.
Now take a partition of $S$ into adjacent $B_4$ and $B_5$ such that $t_4\in B_4, t_5\in B_5$, $G[B_4],G[B_5]$ are connected subgraphs, and $|\{v,w\}\cap B_4|=|\{v,w\}\cap B_5|=1$. 
Then the bags $B_2,\dots ,B_5$ are pairwise adjacent, hence $G$ has a rooted $H$\!-certificate, a contradiction. 
\end{proof}

\begin{corollary}
Let $G$ be a graph with a Kempe coloring $\mathfrak{C}$ of order 5 and let $T$ be a transversal of $\mathfrak{C}$ such that $G[T]$ is connected. 
Then there exists a rooted $K_T$-certificate in $G$. 
\end{corollary}

\begin{proof}
Since $G$ has a Kempe coloring, any pair of transversal vertices is connected by a 2-colored path. Hence, $H(G,\mathfrak{C},T)$ is isomorphic to $K_5$.
Let $H$ be obtained from $H(G,\mathfrak{C},T)$ by removing edges if they exist in $G[T]$, i.\,e.\ $V(H)=T$ and $E(H)=\{st\mid s,t\in T,s\neq t, st\notin E(G)\}$. 
Since $G[T]$ is connected, $|E(G[T])|\geq 4$ and $|E(H)|\leq 6$. 
Thus, $H$ fulfills the conditions of Theorem~\ref{T6} and has property (*). We find a rooted $H$\!-certificate $c$ of $G$. 
It remains to show that $c$ is a rooted $H(G,\mathfrak{C},T)$-certificate of $G$. 
If for $s,t\in T$ the edge $st$ is not in $E(G)$, then $st\in E(H)$ and $B_s=c(s),B_t=c(t)$ are adjacent. 
Otherwise, $st\in E(G)$, then $B_s,B_t$ are connected in $G$ by the edge $st$. 
\end{proof}

\section{Concluding Remarks}
\label{S6}

Assuming that $K_5$ has property (*), then any graph $G$ with a $5$-coloring $\mathfrak{C}$ and a transversal $T$ such that the routing graph $H(G,\mathfrak{C},T)$ is a complete graph on $5$ vertices will have a rooted $K_5$-minor. In particular, these graphs won't be able to be planar and they will have an unrooted $K_5$-minor. 
We conclude with two remarks that such graphs are indeed not planar and have a $K_5$-minor even in the case that $K_5$ may not have property (*).
The problem whether $K_5$ has property (*) remains open. 

\begin{remark}\label{R1}
Let $G$ be a graph with a 5-coloring $\mathfrak{C}$ and let $T$ be a transversal of $\mathfrak{C}$. 
If for each distinct $s,t\in T$ there is a 2-colored path from $s$ to $t$ in $G$, then $G$ is not planar. 
\end{remark}

\begin{proof}
On the contrary assume that $G$ is planar, and, again, we may assume that $G$ is chosen with $|V(G)|+|E(G)|$ minimal, implying that for
all $A\neq B$ from $\mathfrak{C}$, $G[A \cup B]$ has a single nontrivial component which induces a path between the unique vertices $a \in A\cap  T, b \in B \cap T$. 
Consider a drawing of $G$ into the plane. Then each of the ten 2-colored paths between the transversal vertices can be considered as a Jordan curve of a plane drawing of $K_5$ on $T$ with crossings. 
Evoke the Tutte-Hanani-Theorem~\cite{TutteHanani} which states that 
in any planar representation of a non-planar graph $G$ there are two nonadjacent edges whose crossing number is odd.
Since $K_5$ is non-planar, there must be two of the Jordan curves with different end vertices crossing and such a crossing is always a vertex of $G$. 
But then, these two Jordan curves share an end vertex in the same color as the crossing vertex, contradiction. 
\end{proof}

\begin{remark}\label{R3}
Let $G$ be a graph with a 5-coloring $\mathfrak{C}$ and let $T$ be a transversal of $\mathfrak{C}$. 
If for all $s\neq t$ from $T$ there is a 2-colored path from $s$ to $t$ in $G$, then $G$ has $K_5$ as minor. 
\end{remark}

\begin{proof}
Assume that $G$ does not contain $K_5$ as minor and let $G$ be chosen as a counterexample minimizing $|V(G)|+|E(G)|$. 
If $G$ was not 3-connected, then there would be a separator $S$ with $|S|\leq 2$ and a component $C$ of $G-S$ containing at least three vertices from $T$. 
Repeating the same arguments as in the proof of Theorem~\ref{T3}, we can reduce $G$ to a smaller counterexample. 
Hence, we assume that $G$ is 3-connected. 

Denote by $G^+$ the graph obtained from $G$ by repeatedly adding edges as long as the resulting graph does not contain a $K_5$-minor. 
Then, $G^+$ is a 3-connected maximal $K_5$-minor-free graph by construction. 
By a famous result of Wagner~\cite{Wagner}, $G^+$ is a 3-clique-sum of maximal planar graphs or the 8-vertex Wagner graph. 
Since $G$ is not planar by Remark~\ref{R1} and not the 8-vertex Wagner graph ($G$ has minimum degree at least $4$), 
there is a clique $S$ in $G^+$ with $|S|=3$ that separates $G^+$. 
Since $G$ is a spanning subgraph of $G^+$, $S$ also separates $G$.

Let  $A_x$ denote the member of $\mathfrak{C}$ containing $x$ with $x\in V(G)$. 
Then there are $s$ ($s\geq 2$) distinct $x_1,x_2,\dots, x_s\in T$ such that $A_{x_i}\cap S=\emptyset$ for $i\in\{1,\dots,s\}$. 
Again as in the proof of Theorem~\ref{T3}, there is one component $C$ of $G-S$ containing all $x_i$, $i\in\{1,\dots,s\}$. 

Let $X:=V(G) \setminus (V(C) \cup S)$ and assume first that there is $y\in T$ such that $|A_y\cap S|\geq 2$.
Let $G'$ be obtained from $G$ by contracting $Y:=X\cup (S\cap A_y)$ to a single vertex $w$. 
For $A \in \mathfrak{C}$ set $A':=(A \setminus Y) \cup \{w\}$ if $A=A_y$ and $A':=A \setminus Y$ otherwise, so that $\mathfrak{C}':=\{A'\mid A \in \mathfrak{C}\}$ is a coloring of $G'$.
For $z \in T$, set $z':=w$ if $z \in Y\cap A_y$,  $z':=z_0$ with $z_0\in S\setminus A_y$ uniquely determined if $z \in Y\setminus A_y$ and $z':=z$ otherwise, so that $T':=\{z'\mid z \in T\}$ is a transversal of $\mathfrak{C}'$,
and $H':=H(G',\mathfrak{C}',T')$ is a complete graph on $T'$.
By the choice of $G$, $G'$ has a $K_5$-minor and so has $G$ because $G'$ is a minor of $G$, contradiction. 

Thus, there are distinct $y_1,y_2,y_3\in T$ such that $|A_{y_i}\cap S|=1$ with $i\in\{1,2,3\}$.
If $X=\{d\}$, then $d$ is not in $T$ and $d$ cannot be a vertex of a 2-colored path of $G$, contradiction. 
Thus, $X$ consists of at least two vertices with degree at least 3 in $G[X\cup S]$. 
If there was no cycle in $G[X\cup S]$, then $G[X\cup S]$ would be  a tree. A leaf of this tree would be a vertex from $S$, contradiction because a tree with at least two vertices of degree at least three has at least four leaves. 

Hence, there is a cycle $D$ in $G[X\cup S]$ and by the 3-connectedness of $G$, there are three vertex disjoint $s_i$,$c_i$-paths $P_i$, $i\in\{1,2,3\}$ in $G$ such that $V(P_i)\cap S=\{s_i\}$ and $V(P_i)\cap V(D)=\{c_i\}$ (possibly $s_i=c_i$). 
It is easy to see that $V(P_i)\subseteq X\cup S$. 
Denote by $D_i$ the subpath of $D$ from $c_i$ to $c_{i+1}$ missing $c_{i+2}$ (indices modulo 3). 
Let $G'$ be obtained from $G[V(C)\cup D\cup \bigcup_{i=1}^3 V(P_i)]$ by contracting $V(P_i)\cup (V(D_i)\setminus\{c_{i+1}\})$ to a single vertex $w_i$ for $i\in\{1,2,3\}$. 
Observe that $w_1w_2w_3$ is a triangle in $G'$. 
For $A \in \mathfrak{C}$ set $A':=(A \cap V(C)) \cup \{w_i\}$ if $s_i\in A$ and $A':=A \cap V(C)$ otherwise, so that $\mathfrak{C}':=\{A'\mid A \in \mathfrak{C}\}$ is a coloring of $G'$.
For $z \in T$, set $z':=w_i$ if $z \notin V(C)$ with $z\in A_{s_i}$ for suitable $i\in\{1,2,3\}$,  and $z':=z$ otherwise, so that $T':=\{z'\mid z \in T\}$ is a transversal of $\mathfrak{C}'$.
It is straightforward to check that $H':=H(G',\mathfrak{C}',T')$ is a complete graph on $T'$.
By the choice of $G$, $G'$ has a $K_5$-minor and so has $G$ because $G'$ is a minor of $G$, contradiction. 
\end{proof}

\printbibliography

\end{document}